\documentclass[12pt]{amsart}

\usepackage[utf8]{inputenc}
\usepackage[T1]{fontenc}

\usepackage{amsmath, amsfonts, amsthm, amssymb}
\usepackage[top=1in, bottom=1in, left=1in, right=1in]{geometry}
\usepackage[backref=page, colorlinks=true, allcolors=blue]{hyperref}
\usepackage[alphabetic,abbrev,lite,backrefs]{amsrefs} % for bibliography 
\usepackage{tikz}
\usepackage{xcolor}
\usepackage{changepage}
\usepackage{mathrsfs}
\usepackage{tikz-cd}
\usepackage{xargs}
\usepackage{mathtools} % just for defining the symbol :=

\theoremstyle{plain}
	\newtheorem{theorem}{Theorem}[section]
	\newtheorem{lemma}[theorem]{Lemma}

\theoremstyle{definition}
    \newtheorem{defn}[theorem]{Definition}
    \newtheorem{example}[theorem]{Example}
    
\theoremstyle{remark}
	\newtheorem{remark}[theorem]{Remark}
	\newtheorem*{remark*}{Remark}

\usepackage{mathtools}
\newcommand\SetSymbol[1][]{\nonscript\:#1\vert\allowbreak\nonscript\:\mathopen{}}
\providecommand\given{} % to make it exist
\DeclarePairedDelimiterX\Set[1]\{\}{\renewcommand\given{\SetSymbol[\delimsize]}#1}

\newcommand{\ms}[1]{\mathscr{#1}}
\newcommand{\mc}[1]{\mathcal{#1}}
\newcommand{\coker}[0]{\text{coker }}

\newcommand{\ideal}[1]{\langle #1 \rangle}

\def\Z{{\mathbb Z}}

\def\P{{\mathbb P}}

\def\O{{\mathcal O}}

\DeclareMathOperator{\Proj}{Proj}

\DeclareMathOperator{\Sym}{Sym}

\DeclareMathOperator{\rank}{rank}
\DeclareMathOperator{\depth}{depth}

\DeclareMathOperator{\Pic}{Pic}
\DeclareMathOperator{\Cox}{Cox}

\newcommand{\tensor} {\otimes}

\newcommand{\ds}{\displaystyle}

\title{Virtual criterion for generalized Eagon-Northcott complexes}
\thanks{Booms-Peot is supported by NSF GRFP fellowship DGE-1747503 and by the Graduate School and OVCRGE at UW-Madison with funding from the Wisconsin Alumni Research Foundation. Cobb is partially supported by a DoD NDSEG fellowship.}
\author{Caitlyn Booms-Peot}
\address{Department of Mathematics, University of Wisconsin, Madison, WI}
\email{cbooms@math.wisc.edu}

\author{John Cobb}
\address{Department of Mathematics, University of Wisconsin, Madison, WI}
\email{jcobb2@math.wisc.edu}

\begin{document}

\maketitle
\vspace{-0.3in}
\begin{abstract}
    Given any map of finitely generated free modules, Buchsbaum and Eisenbud define a family of generalized Eagon-Northcott complexes associated to it \cite{BE75}. We give sufficient criterion for these complexes to be virtual resolutions, thus adding to the known examples of virtual resolutions, particularly those not coming from minimal free resolutions.
\end{abstract}
% \tableofcontents

\setcounter{section}{1}
\section*{Introduction}

The Eagon-Northcott complex of a matrix has been an object of interest since its introduction in \cite{Eagon1962IdealsDB}, where the authors showed that it is a minimal free resolution of the ideal of maximal minors of the matrix if the depth of this ideal is the greatest possible value. In 1975, Buchsbaum and Eisenbud described a family of generalized Eagon-Northcott complexes associated to a map of free modules, which are free resolutions if the depth of the ideal of maximal minors of the matrix is as large as possible \cites{BE75,Eisenbud2004}. Our main result provides an analogue in the setting of virtual resolutions. 

Let us recall the setup of virtual resolutions. Let $X$ be a smooth projective toric variety over an algebraically closed field $k$, and let $S$ be its $\Pic(X)$-graded Cox ring, which is a multigraded polynomial ring with irrelevant ideal $B$, as defined in \cite{Cox1993TheHC}. In 2017, Berkesch, Erman, and Smith formalized the notion of \textit{virtual resolutions} as a natural analogue to minimal graded free resolutions for smooth projective toric varieties \cite{BerkeschErman2017VirtualRF}.\footnote{Even without a formal definition, this notion appeared in the literature prior \cites{Maclagan2003MultigradedCR, Ein2018SyzygiesOP}.} In exchange for allowing some higher homology supported on the irrelevant ideal $B$, virtual resolutions tend to be shorter and to better capture geometrically meaningful properties of $S$-modules, such as unmixedness, well-behavedness of deformation theory, and regularity of tensor products. 

Despite their utility, there continues to be a lack of families of examples of virtual resolutions. Our work adds a new method for explicitly constructing virtual resolutions, and it adds to the growing literature on developing virtual analogues of classical homological results \cites{Berkesch2020HomologicalAC ,YangVirtualMonomial, harada2021virtual,Dowd2019VIRTUALRO, Loper2019WhatMA, kenshur2020virtually,GAO2021106473,bruce2021characterizing}. 
In particular, we show when the Eagon-Northcott and Buchsbaum-Rim complexes (and the other complexes in this generalized family) are virtual resolutions.

Let $X$ be a smooth projective toric variety with $S=\Cox(X)$ and irrelevant ideal $B$. Our main theorem is as follows:

\begin{theorem}\label{eisenbudvirtual}
Let $\varphi: F \to G$ be a $\Pic(X)$-graded map of free $S$-modules of ranks $f \geq g$ where $I_m(\varphi)$ denotes the ideal of $m\times m$ minors of $\varphi$, and let 
$$\mathcal{C}^i: 0 \longrightarrow F_e \overset{\varphi_e}{\longrightarrow} F_{e-1} \overset{\varphi_{e-1}}{\longrightarrow} \cdots \overset{\varphi_2}{\longrightarrow} F_1 \overset{\varphi_1}{\longrightarrow} F_0$$
be one of the generalized Eagon-Northcott complexes of $\varphi$ with $i\geq -1$ (see \S 2.2).
Then $\mc{C}^i$ is a virtual resolution whenever $\depth (I_m(\varphi) : B^\infty) \geq f-m+1$ for $f-e+1\leq m \leq g$. 
\end{theorem}

Notably, if $\depth (I_g(\varphi): B^\infty)\geq f-g+1$, then each of the complexes $\mc{C}^i$ with $-1\leq i \leq f-g+1$ (including the Eagon-Northcott and Buchsbaum-Rim) is a virtual resolution of length $f-g+1$ by Theorem \ref{eisenbudvirtual}. The proof of Theorem \ref{eisenbudvirtual} involves a combination of the corresponding methods from \cite{Eisenbud2004}, Loper's criterion for when a complex is a virtual resolution \cite{Loper2019WhatMA}, and several concrete arguments to address how these auxiliary results interact with saturation. 

To illustrate its use, consider the graph of the twisted cubic in the following example.

\begin{example}[Graph of the twisted cubic]\label{twistedcubic}
    Let $X=\P^1\times \P^3$ with $S=k[x_0,x_1,y_0,y_1,y_2,y_3]$. This is $\Pic(X)= \Z^2$-graded where $\deg(x_i) = (1,0)$ and $\deg(y_i)=(0,1)$. The irrelevant ideal $B$ is $\ideal{x_0,x_1} \cap \ideal{y_0,y_1,y_2, y_3}$. Consider the following map of $\Pic(X)$-graded $S$-modules:
    \begin{center}
        \begin{tikzcd}
            \varphi \coloneqq 
    \begin{bmatrix}
    x_0^3 & x_0^2x_1 & x_0x_1^2 & x_1^3\\
    y_0 & y_1 & y_2 & y_3
    \end{bmatrix}: S(-3,-1)^4 \ar[r] & \begin{matrix} S(0,-1)\\
            \oplus\\
            S(-3,0)
            \end{matrix}
        \end{tikzcd}
    \end{center}
    The variety in $X$ defined by $I_2(\varphi)$, the ideal of $2\times 2$ minors of $\varphi$, is precisely the graph of the embedding of the twisted cubic curve defined by $[s:t] \mapsto ([s:t],[s^3:s^2t:st^2:t^3])$.
    The Eagon-Northcott complex $\mc{C}^0$ is computed to be
    \begin{center}
    \begin{tikzcd}
        \mc{C}^0:
        0 \ar[r] &
        \begin{matrix}
                S(-9,-1)\\ \oplus \\ S(-6,-2)\\ \oplus\\ S(-3,-3)
        \end{matrix}
        \ar[r] & 
        \begin{matrix}
                S(-6,-1)^4\\ \oplus \\ S(-3,-2)^4
        \end{matrix}
        \ar[r] &
        S(-3,-1)^6 
        \ar[r, "I_2(\varphi)"] &
        S
    \end{tikzcd}
\end{center}
If $\depth I_2(\varphi) = 3$, then Theorem A2.10c in \cite{Eisenbud2004} would ensure that $\mc{C}^0$ is a free resolution of $S/I_2(\varphi)$. However, since $\ideal{x_0,x_1}$ is a minimal prime of $I_2(\varphi)$, we have $\depth I_2(\varphi) \leq 2$, and one can in fact check that $H_1(\mc{C}^0)\neq 0$ (see Example \ref{degreedexample}). A similar computation shows that the Buchsbaum-Rim complex $\mc{C}^1$ is not a free resolution of $\coker \varphi$. However, since $\depth(I_2(\varphi):B^\infty)=3$, Theorem \ref{eisenbudvirtual} implies that both $\mc{C}^0$ and $\mc{C}^1$ are virtual resolutions. 
\end{example}

This paper is organized as follows: in \S 2, we provide notation and necessary background about virtual resolutions and generalized Eagon-Northcott complexes, in \S3, we give a proof of Theorem \ref{eisenbudvirtual}, and in \S 4, we explore more examples of the utility of our result.

\subsection*{Acknowledgements}
The authors thank Daniel Erman for his valuable insight and guidance and Mahrud Sayrafi for helpful conversations. We also thank our anonymous referees for their valuable suggestions. The computer algebra system Macaulay2 \cite{M2} was used extensively, in particular the \texttt{VirtualResolutions} package \cite{Almousa_2020}.

\section{Background and Notation}

\subsection{Virtual Resolutions} 

On $\P^n$, minimal free resolutions capture geometric properties well because the maximal ideal, which plays a key role in the definition of such resolutions, is the irrelevant ideal. However, on toric varieties such as $\P^1 \times \P^1$, the irrelevant ideal $B=\ideal{x_0,x_1}\cap\ideal{y_0,y_1}$ is strictly contained in the maximal ideal. This inequality results in the minimal free resolution containing algebraic structure that is geometrically irrelevant. By focusing on the irrelevant ideal, virtual resolutions are better able to represent important geometric information.

From now on, let $X$ be a smooth projective toric variety with $S=\Cox(X)$ and irrelevant ideal $B$.\footnote{$S$ and $B$ are generalizations of the homogeneous coordinate ring and its maximal ideal, respectively. For more precise definitions, see \cite{Cox1993TheHC}*{\S 5.1}, where $\Cox(X)$ is called the total coordinate ring.} Given that there is a correspondence between $\Pic(X)$-graded $B$-saturated $S$-modules $M$ and sheaves $\widetilde{M}$ on $X$ (for more details see \S 5.2, \cite{Cox}; a generalization can be found in \cite{mustata_vanishing_2002}), allowing for some ``irrelevant homology'' in our complexes stands to make them shorter and closer linked to the geometric situation. This motivates the following definition.

\begin{defn}[\cite{BerkeschErman2017VirtualRF}]\label{VirtualResolution}
A complex $\mc{C}: \cdots \overset{\varphi_3}{\longrightarrow} F_2 \overset{\varphi_2}{\longrightarrow} F_1 \overset{\varphi_1}{\longrightarrow} F_0$ of $\Pic(X)$-graded free $S$-modules is called a \textbf{virtual resolution} of a $\Pic(X)$-graded $S$-module $M$ if the corresponding complex $\widetilde{\mc{C}}$ of vector bundles on $X$ is a locally free resolution of the sheaf $\widetilde{M}$. 
\end{defn}

Algebraically, $\mc{C}$ is a virtual resolution if all of the higher homology groups are supported on the irrelevant ideal, i.e. for each $i\geq 1$, $B^nH_i(\mc{C})=0$ for some $n$. Note that all exact complexes are virtual resolutions, but not all virtual resolutions are exact, since they allow for a specific type of homology.

Our proof of Theorem \ref{eisenbudvirtual} utilizes a result from \cite{Loper2019WhatMA} which provides a virtual analogue of Buchsbaum and Eisenbud's famous criterion for checking the exactness of a complex without directly computing the homology \cite{Buchsbaum1973WhatMA}. Given a $\Pic(X)$-graded map of free $S$-modules $\varphi: F \to G$, we can choose a matrix representation for $\varphi$. Let $I_m(\varphi)\subseteq S$ be the ideal generated by the $m \times m$ minors of this matrix, where, by convention, $I_m(\varphi)=S$ for $m\leq 0$. Note that these ideals of minors are Fitting invariants of $\coker \varphi$, independent of the choice of matrix representation \cite[\S 20.2]{Eisenbud2004}. We define the rank of $\varphi$ to be $\rank(\varphi):= \max\Set*{m \given I_m(\varphi)\neq 0}$, and we set $I(\varphi):= I_{\rank(\varphi)}(\varphi)$ to be the corresponding ideal of minors, which will play a key role in our study of $\varphi$. Loper's criterion for a complex to be a virtual resolution is as follows.

\begin{theorem}[\cite{Loper2019WhatMA}] \label{loper}
Suppose \[\mc{C}: 0 \longrightarrow F_e \overset{\varphi_e}{\longrightarrow} F_{e-1} \longrightarrow \cdots \overset{\varphi_2}{\longrightarrow} F_1 \overset{\varphi_1}{\longrightarrow} F_0\] is a $\Pic(X)$-graded complex of free $S$-modules. Then $\mc{C}$ is a virtual resolution if and only if both of the following conditions are satisfied for each $j=1,\dots, e$:
\begin{enumerate}
    \item[(i)] $\rank \varphi_j +\rank \varphi_{j+1} = \rank F_j$ \qquad (taking $\varphi_{e+1}=0$),
    \item[(ii)] $\depth (I(\varphi_j):B^\infty)\geq j$
\end{enumerate}
\end{theorem}

Therefore, in order for a complex in the toric setting to be a virtual resolution, we only need to consider the depth of the $B$-saturation of the maximal ideals of minors of the complex's differentials. By convention, the depth of the unit ideal is infinity, so that condition (ii) above is satisfied if $(I(\varphi_j):B^\infty)=S$.

\subsection{Generalized Eagon-Northcott Complexes}
Let $\varphi: F \to G$ be any map of free $R$-modules where $f=\rank F \geq g=\rank G$. Associated to $\varphi$ is a family $\{\mc{C}^i\}_{i\in \Z}$ of \textit{generalized Eagon-Northcott complexes} defined in \cite[\S A2.6]{Eisenbud2004} by splicing together linear strands of particular Koszul complexes. Two of the most important of these complexes are $\mc{C}^0$, the Eagon-Northcott complex, and $\mc{C}^1$, the Buchsbaum-Rim complex, which are shown below. Here, $\bigwedge^d F$ denotes the $d^{\text{th}}$ exterior power of $F$ and $(\Sym_d(G))^*$ denotes the dual of the $d^{\text{th}}$ symmetric power of $G$.
%\vspace{-0.75em}

\begin{center}\footnotesize
\begin{tikzcd}
    \mc{C}^0:  0 \ar[r] &[-1.5em] \bigwedge^f F \otimes (\Sym_{f-g}(G))^* \ar[r] &[-1.25em] \cdots \ar[r] &[-1.25em] \bigwedge^{g+1} F \otimes (\Sym_1(G))^* \ar[r] &[-1.25em] \bigwedge^g F\otimes (\Sym_0(G))^* \ar[r, "\wedge^g \varphi"] & R \ar[r] &[-1em] 0  
\end{tikzcd}
\begin{tikzcd}
    \mc{C}^1:  0 \ar[r] &[-1.8em] \bigwedge^f F \otimes (\Sym_{f-g-1}(G))^* \ar[r] &[-1.8em] \cdots\hspace{-0.1em} \ar[r] &[-1.8em] \bigwedge^{g+2} F \otimes (\Sym_1(G))^* \ar[r] &[-1.8em] \bigwedge^{g+1} F \otimes (\Sym_0(G))^* \ar[r, "\epsilon"] &[-1.62em] F \ar[r, "\varphi"] &[-1.62em] G \ar[r] &[-1.8em] 0
\end{tikzcd}
\end{center}

 This family of complexes has unique properties, leading to their use in a wide variety of situations \cites{Schreyer1986SyzygiesOC, Gruson1983OnAT, Zamaere2013TensorCM}. By their construction, each complex is dual to another: $\mc{C}^i$ is dual to $\mc{C}^{f-g-i}$. The most interesting complexes in this family are those $\mc{C}^i$ with $-1\leq i \leq f-g+1$. These complexes are of length $f-g+1$ and are generically free resolutions \cite[Theorem A2.10c]{Eisenbud2004}.
 The $\mc{C}^i$ can be described explicitly, and they preserve homogeneity in the case that $R$ is a (multi)graded ring.
 
 \begin{example}\label{2x4example}
 Consider the homogeneous map
 \begin{center}
    \begin{tikzcd}
        \varphi \coloneqq 
    \begin{bmatrix}
    x_{11} & x_{12} & x_{13} & x_{14}\\
    x_{21} & x_{22} & x_{23} & x_{24}
    \end{bmatrix}:
    S(-1,-1)^4 \ar[r] & 
    \begin{matrix}
    S(0,-1)\\
    \oplus\\
    S(-1,0)
    \end{matrix}
    \end{tikzcd}
\end{center}
where $S=k[x_{ij}]=\Cox(\P^3\times\P^3)$ is bigraded with $\deg(x_{1j})=(1,0)$ and $\deg(x_{2j})=(0,1)$.
Let $m_{ij}$ be the $2\times 2$ minor of $\varphi$ involving columns $i$ and $j$, so that $I_2(\varphi) = \ideal{m_{ij}}$. The Eagon-Northcott complex $\mc{C}^0$ and Buchsbaum-Rim complex $\mc{C}^1$ are given by
\begin{center}\footnotesize
    \begin{tikzcd}[ampersand replacement=\&]
        \mc{C}^0:
        \begin{matrix}
        S(-3,-1)\\
        \oplus\\
        S(-2,-2)\\
        \oplus\\
        S(-1,-3)
        \end{matrix} \arrow{r}{\left[\begin{smallmatrix}
        -x_{14} & -x_{24} & 0\\
        x_{13} & x_{23} & 0\\
        -x_{12} & -x_{22} & 0\\
        x_{11} & x_{21} & 0\\
        0 & -x_{14} & -x_{24}\\
        0 & x_{13} & x_{23}\\
        0 & -x_{12} & -x_{22}\\
        0 & x_{11} & x_{21}
        \end{smallmatrix}\right]} \&[4em]
        \begin{matrix}
        S(-2,-1)^4\\
        \oplus\\
        S(-1,-2)^4
        \end{matrix} \arrow{r}{\left[\begin{smallmatrix}
        x_{13} & x_{14} & 0 & 0 & x_{23} & x_{24} & 0 & 0\\
        -x_{12} & 0 & x_{14} & 0 & -x_{22} & 0 & x_{24} & 0\\
        x_{11} & 0 & 0 & x_{14} & x_{21} & 0 & 0 & x_{24}\\
        0 & -x_{12} & -x_{13} & 0 & 0 & -x_{22} & -x_{23} & 0\\
        0 & x_{11} & 0 & -x_{13} & 0 & x_{21} & 0 & -x_{23}\\
        0 & 0 & x_{11} & x_{12} & 0 & 0 & x_{21} & x_{22}
        \end{smallmatrix}\right]} \&[14.5em]
        S(-1,-1)^6 \arrow{r}{I_2(\varphi)} \&
        S\\
    \end{tikzcd}
    \begin{tikzcd}[ampersand replacement=\&]
        \mc{C}^1:
        \begin{matrix}
        S(-3,-2)\\
        \oplus\\
        S(-2,-3)
        \end{matrix} \arrow{r}{\left[\begin{smallmatrix}
        -x_{14} & -x_{24}\\
        x_{13} & x_{23}\\
        -x_{12} & -x_{22}\\
        x_{11} & x_{21}
        \end{smallmatrix}\right]} \&[4em]
        S(-2,-2)^4 \arrow{r}{\left[\begin{smallmatrix}
        m_{23} & m_{24} & m_{34} & 0\\
        -m_{13} & -m_{14} & 0 & m_{34}\\
        m_{12} & 0 & -m_{14} & -m_{24}\\
        0 & m_{12} & m_{13} & m_{23}
        \end{smallmatrix}\right]} \&[10em]
        S(-1,-1)^4 \arrow{r}{\varphi} \&
       \begin{matrix}
        S(0,-1)\\
        \oplus\\
        S(-1,0)
        \end{matrix}
    \end{tikzcd}
\end{center}
These are minimal free resolutions of $S/I_2(\varphi)$ and $\coker\varphi$, respectively.
\end{example}

%%%%%%%%%%%%%%%%%%%%%%%%%%%%%%%%
\section{Main Result and Proof}
%%%%%%%%%%%%%%%%%%%%%%%%%%%%%%%

Our main result gives a sufficient criterion for the generalized Eagon-Northcott complexes $\{\mc{C}^i\}_{i\geq-1}$ of a $\Pic(X)$-graded map of free $S$-modules $\varphi: F \to G$ with ranks $f\geq g$ to be virtual resolutions. This is a virtual analogue of \cite[Theorem A2.10c]{Eisenbud2004} (see \S 2.2) and will require Theorem \ref{loper} to prove. Note that while one could apply Theorem \ref{loper} directly to a given $\mc{C}^i$ to determine if it is a virtual resolution, this would require checking $\depth(I(\varphi_j):B^\infty)$ for each differential in $\mc{C}^i$. The main utility of Theorem \ref{eisenbudvirtual} is that it enables one to determine that the entire family $\{\mc{C}^i\}_{i\geq -1}$ consists of virtual resolutions by only checking $\depth(I_m(\varphi):B^\infty)$ for particular $m$ and the single map $\varphi$. The proof of Theorem \ref{eisenbudvirtual} will require the following lemma, a virtual analogue of \cite[Theorem A2.10b]{Eisenbud2004}.

\begin{lemma}\label{part b}
For $i\geq -1$, let
\[\mathcal{C}^i: 0 \longrightarrow F_e \overset{\varphi_e}{\longrightarrow} F_{e-1} \overset{\varphi_{e-1}}{\longrightarrow} \cdots \overset{\varphi_2}{\longrightarrow} F_1 \overset{\varphi_1}{\longrightarrow} F_0\]
be one of the generalized Eagon-Northcott complexes of $\varphi$, and let $r(j)= \ds\sum_{\ell=j}^e (-1)^{\ell-j} \rank F_\ell$. Then for each $1\leq j \leq e$, we have $\rank \varphi_j \leq r(j)$ and $(I_{r(j)}(\varphi_j):B^\infty)$ is contained in and has the same radical as the ideal $(I_{s(j)}(\varphi):B^\infty)$, where $s(j) = \min(g,f-j+1)$.
\end{lemma}
\begin{proof}
    This will follow readily from \cite[Theorem A2.10b]{Eisenbud2004} once we understand how taking radicals plays with saturation. We claim that for ideals $I,J$ in a Noetherian ring $R$, $\sqrt{(I:J^\infty)}=(\sqrt{I}:J)$. (This is likely known to experts, but we provide a proof for completeness.) For the first containment, let $r \in \sqrt{(I:J^\infty)}$. Then there exists $a$ such that $r^a \in (I:J^\infty)$ so there exists $b$ such that $r^aJ^b \subseteq I$. For $c=\max\{a,b\}$ and any $j \in J$, we have $(rj)^c= r^{c-a}(r^aj^c) \in r^aJ^c \subseteq I$. Thus, $rJ \subseteq \sqrt{I}$, so $r \in (\sqrt{I}:J)$. Conversely, let $r \in (\sqrt{I}:J)$, and suppose $J = \ideal{j_1,\dots,j_s}$. Then $rj_i \in \sqrt{I}$ and we can choose $a\gg 0$ such that $r^aj_i^a\in I$ for each $i$. Since $J^{sa} \subseteq \ideal{j_i^a,\dots,j_s^a}$, we have $r^a J^{sa} \subseteq I$. Thus, $r^a \in (I:J^{\infty})$ and $r \in \sqrt{(I:J^\infty)}$.
    
    By \cite[Theorem A2.10b]{Eisenbud2004}, $\rank \varphi_j \leq r(j)$, $I_{r(j)}(\varphi_j) \subseteq I_{s(j)}(\varphi)$, and  $\sqrt{I_{r(j)}(\varphi_j)} = \sqrt{I_{s(j)}(\varphi)}$. Therefore, $(I_{r(j)}(\varphi_j):B^\infty) \subseteq (I_{s(j)}(\varphi):B^\infty)$, and the above argument gives that $\sqrt{(I_{r(j)}(\varphi_j):B^\infty)} = (\sqrt{I_{r(j)}(\varphi_j)}: B) = (\sqrt{I_{s(j)}(\varphi)}:B) = \sqrt{(I_{s(j)}(\varphi): B^\infty)}$.
\end{proof}

\begin{proof}[Proof of Theorem \ref{eisenbudvirtual}]
    First, note that by the construction of the $\mc{C}^i$ the length satisfies $f-g+1\leq e\leq f$. To show that $\mc{C}^i$ is a virtual resolution, we show that (i) and (ii) from Theorem \ref{loper} hold for all $j=1,\dots, e$. The first property to show is that $\rank \varphi_j+\rank \varphi_{j+1} = \rank F_j$. Since 
    \begin{align*}
        r(j+1) + r(j) &= \left(\rank F_{j+1} - \rank F_{j+2} + \cdots + (-1)^{e-(j+1)}\rank F_e \right) +\cdots \\
        &\hspace{0.5in} \cdots + \left( \rank F_{j} - \rank F_{j+1} + \cdots + (-1)^{(e-j)}\rank F_e \right)\\
        &= \rank F_{j},
    \end{align*}
    condition (i) will follow if $\rank \varphi_j = r(j)$. By Lemma \ref{part b}, we have $\rank \varphi_j \leq r(j)$. Suppose instead that $\rank \varphi_j < r(j)$. Then, since $\rank \varphi_j$ is the largest ideal of minors of $\varphi_j$ that is nonzero, we have that $I_{r(j)}(\varphi_j) = 0$. Note that $\depth (I_{r(j)}( \varphi_j) : B^\infty) = \depth (I_{s(j)}(\varphi) : B^\infty)$ since the depth of ideals is preserved under taking radicals (\cite[Cor. 17.8b]{Eisenbud2004}) and these ideals have the same radical by Lemma \ref{part b}. Thus, $\depth (I_{s(j)}(\varphi) : B^\infty) = 0$. However, since $f-e+1 \leq s(j) \leq g$ for each $j$, this contradicts our assumption that $\depth (I_{s(j)}(\varphi) : B^\infty)\geq f-s(j)+1 \geq 1$. Thus, $\rank \varphi_j = r(j)$ for each $j=1,\dots, e$.
    
    For (ii), we need to show that $\depth(I(\varphi_j):B^\infty)\geq j$ for each $j=1,\dots,e$. Since $\rank \varphi_j = r(j)$ by the above argument, we have $I_{r(j)}(\varphi_j) = I_{\rank \varphi_j}(\varphi_j) = I(\varphi_j)$. Saturating then gives $(I_{r(j)}(\varphi_j):B^\infty) = (I(\varphi_j):B^\infty)$. Since $\sqrt{(I_{r(j)}(\varphi_j):B^\infty)} = \sqrt{(I_{s(j)}(\varphi):B^\infty)}$ by Lemma \ref{part b}, we have $\depth (I(\varphi_j):B^\infty) = \depth(I_{r(j)}(\varphi_j):B^\infty) =  \depth(I_{s(j)}(\varphi):B^\infty)$. In the case that $s(j)=g$, our assumption gives that $\depth (I(\varphi_j):B^\infty) = \depth(I_{s(j)}(\varphi):B^\infty) \geq f-g+1 \geq j$ since $g\leq f-j+1$.
    Otherwise, if $s(j)=f-j+1$, then $\depth (I(\varphi_j) : B^\infty)\geq f-(f-j+1)+1=j$ since $f-e+1\leq s(j)\leq g$. In any case, $\depth (I(\varphi_j):B^\infty) \geq j$. Thus, $\mc{C}^i$ is a virtual resolution by Theorem \ref{loper}.
\end{proof}

\begin{remark} \label{remark}
If $\depth (I_g(\varphi): B^\infty)\geq f-g+1$, then the complexes $\mc{C}^i$ with $-1\leq i \leq f-g+1$ are virtual resolutions of length $f-g+1$ by Theorem \ref{eisenbudvirtual}. In particular, $\mc{C}^{-1}$ is a virtual resolution of $\bigwedge^{f-g+1}(\coker \varphi^*)$, the Eagon-Northcott complex $\mc{C}^0$ is a virtual resolution of $S/I_g(\varphi)$, the Buchsbaum-Rim complex $\mc{C}^1$ is a virtual resolution of $\coker \varphi$, and $\mc{C}^i$ is a virtual resolution of $\Sym_i(\coker \varphi)$ when $1< i \leq f-g+1$.
\end{remark}

\section{Examples of Virtual Resolutions}

Theorem \ref{eisenbudvirtual} gives a new way of producing virtual resolutions, especially those which are not themselves free resolutions. One aspect that is new is that our virtual resolutions are not constructed by paring down a minimal free resolution, as is the case in the construction of a virtual resolution of a pair in \cite{BerkeschErman2017VirtualRF}. In particular, Theorem \ref{eisenbudvirtual} tells us that we can restrict our search to finding $\Pic(X)$-graded maps $\varphi$ of free $S$-modules where $\depth I_g(\varphi) < f-g+1$ but $\depth(I_g(\varphi):B^\infty)\geq f-g+1$. Note that while the depth of $I_g(\varphi)$ is bounded above by $f-g+1$, saturating allows for the depth to increase, potentially to infinity if $(I_g(\varphi):B^\infty)=S$. Under these conditions, we know that the Eagon-Northcott and Buchsbaum-Rim, along with the other complexes $\mc{C}^i$ with $-1\leq i\leq f-g+1$, are virtual resolutions which may not be free resolutions. If, in addition, $\depth(I_m(\varphi):B^\infty)\geq f-m+1$ for all $1\leq m\leq g$, then Theorem \ref{eisenbudvirtual} ensures that the remaining $\mc{C}^i$ with $i>f-g+1$ are virtual resolutions as well.

\begin{example}[Graph of a Hirzebruch]
Consider the Hirzebruch surface $\ms{H}_1=\P(\O_{\P^1} \oplus \O_{\P^1}(1))$ with projective coordinates $[x_0:x_1:x_2:x_3]$ and Cox ring $C=k[x_0,x_1,x_2,x_3]$ with $\deg(x_0)=\deg(x_2)=(1,0)$, $\deg(x_3)=(0,1)$, and $\deg(x_1)=(-1,1)$ (see \cite[pg. 112]{Cox}). The smallest ample line bundle on the Hirzebruch is $\O_{\ms{H}_1}(1,1)$, and its global sections are given by
\[ H^0(\ms{H}_1,\O_{\ms{H}_1}(1,1)) = C_{(1,1)} = k\langle x_0x_3,x_2x_3, x_0^2x_1,x_0x_1x_2,x_1x_2^2\rangle. \]
That is, its global sections are spanned by all combinations of generators whose total bi-degree is $(1,1)$. Since $\O_{\ms{H}_1}(1,1)$ is actually globally generated by these sections, we get a map
$\varphi: \ms{H}_1 \to \Proj(C_{(1,1)}) = \P^4$ defined by $\begin{bmatrix} x_0x_3 & x_2x_3 & x_0^2x_1 & x_0x_1x_2 & x_1x_2^2 \end{bmatrix}$.

Let $X=\ms{H}_1\times \P^4$, so that our Cox ring is $S = C \tensor k[z_0,z_1,z_2,z_3,z_4]$ with grading $\deg(x_0)=\deg(x_2) = (1,0,0)$, $\deg(x_3)=(0,1,0)$, and  $\deg(x_1)=(-1,1,0)$ inherited from $C$ and $\deg(z_i)=(0,0,1)$. This has irrelevant ideal $B= \langle x_0,x_2 \rangle \cap \langle x_1,x_3 \rangle \cap \langle z_0,z_1,z_2,z_3,z_4\rangle$. Consider the $\Pic(X)$-graded map
\begin{center}
    \begin{tikzcd}
    \psi \coloneqq
        \begin{bmatrix}
            x_0x_3 & x_2x_3 & x_0^2x_1 & x_0x_1x_2 & x_1x_2^2\\
            z_0 & z_1 & z_2 & z_3 & z_4
        \end{bmatrix}
        : S(-1,-1,-1)^5 \ar[r] &
        \begin{matrix}
            S(0,0,-1)\\
            \oplus\\
            S(-1,-1,0)
        \end{matrix}
    \end{tikzcd}
\end{center}
The variety in $X$ defined by $I_2(\psi)$ is precisely the graph of $\varphi$ constructed above. Then $\depth I_2(\psi) = 2$ and $\depth(I_2(\psi):B^\infty) = 4$ so by Remark \ref{remark}, the Eagon-Northcott complex $\mc{C}^0$ and the Buchsbaum-Rim complex $\mc{C}^1$ are virtual resolutions of $S/I_2(\psi)$ and $\coker \psi$, respectively:

\begin{center} \tiny
\begin{tikzcd}
    \mc{C}^0: 0 \ar[r] & 
    \begin{matrix}
    S(-4,-4,-1)\\
    \oplus \\
    S(-3,-3,-2)\\
    \oplus\\
    S(-2,-2,-3)\\
    \oplus\\
    S(-1,-1,-4)
    \end{matrix} \ar[r] &
    \begin{matrix}
    S(-3,-3,-1)^5\\
    \oplus\\
    S(-2,-2,-2)^5\\
    \oplus\\
    S(-1,-1,-3)^5
    \end{matrix}\ar[r] &
    \begin{matrix}
    S(-2,-2,-1)^{10}\\
    \oplus\\
    S(-1,-1,-2)^{10}
    \end{matrix} \ar[r] &
    S(-1,-1,-1)^{10} \ar[r,"I_2(\psi)"] & S\\
    
    \mc{C}^1: 0 \ar[r] & 
    \begin{matrix}
    S(-4,-4,-2)\\
    \oplus\\
    S(-3,-3,-3)\\
    \oplus\\
    S(-2,-2,-4)
    \end{matrix} \ar[r] &
    \begin{matrix}
    S(-3,-3,-2)^5\\
    \oplus\\
    S(-2,-2,-3)^5
    \end{matrix}\ar[r] &
    S(-3,-3,-3)^5 \ar[r] &
    S(-1,-1,-1)^5 \ar[r,"\psi"] & 
    \begin{matrix}
    S(0,0,-1)\\
    \oplus\\
    S(-1,-1,0)
    \end{matrix}
\end{tikzcd}
\end{center}
However, one can compute nonzero elements of $H_1(\mc{C}^0)$ and $H_1(\mc{C}^1)$, so these are virtual resolutions which are not exact.
\end{example}

\begin{example}
Let $X=\P^1\times\P^2\times\P^2$ with $S=k[x_0,x_1,y_0,y_1,y_2,z_0,z_1,z_2]$, which is $\Pic(X)=\Z^3$-graded with $\deg(x_i)=(1,0,0), \deg(y_i)=(0,1,0)$, and $\deg(z_i)=(0,0,1)$ and has $B = \ideal{x_0,x_1}\cap \ideal{y_0,y_1,y_2}\cap \ideal{z_0,z_1,z_2}$. Consider the $\Pic(X)$-graded map:
\begin{center}
    \begin{tikzcd}
    \varphi \coloneqq
        \begin{bmatrix}
            x_0^4 & x_0^3x_1 & x_0^2x_1^2 & x_0x_1^3 & x_1^4\\
            0 & y_0^2 & y_1^2 & y_2^2 & 0\\
            z_0 & z_1 & z_2 & z_1 & z_0
        \end{bmatrix}
        : S(-4,-2,-1)^5 \ar[r] &
        \begin{matrix}
            S(0,-2,-1)\\
            \oplus\\
            S(-4,0,-1)\\
            \oplus\\
            S(-4,-2,0)
        \end{matrix}
    \end{tikzcd}
\end{center}
Here, we have that $\depth I_3(\varphi) = 2$ and $\depth(I_3(\varphi):B^\infty)=3$, so Remark \ref{remark} gives that each $\mc{C}^i$ with $-1\leq i\leq 3$ is a virtual resolution, and one can check that these are not exact. 
\end{example}

% Our final example generalizes Example \ref{twistedcubic} to the degree $d$ Veronese embedding of $\P^1$.

\begin{example}[Graph of the degree $d$ rational normal curve] \label{degreedexample}
Let $X=\P^1 \times \P^d$ with $S = k[x_0,x_1,y_0,\dots,y_d]$ and $d\geq 3$. Then $S$ is $\Pic(X)=\Z^2$-graded with $\deg(x_i)=(1,0)$ and $\deg(y_i) = (0,1)$, and the irrelevant ideal $B$ is $\ideal{x_0,x_1} \cap \ideal{y_0,\dots,y_d}$. Consider the following map of $\Pic(X)$-graded $S$-modules:

\begin{center}
        \begin{tikzcd}
            \varphi \coloneqq 
    \begin{bmatrix}
    x_0^d & x_0^{d-1}x_1 & \cdots & x_0x_1^{d-1} & x_1^d\\
    y_0 & y_1 & \cdots & y_{d-1} & y_d
    \end{bmatrix}: S(-d,-1)^{d+1} \ar[r] & \begin{matrix} S(0,-1)\\
            \oplus\\
            S(-d,0)
            \end{matrix}
        \end{tikzcd}
    \end{center}

The variety in $X$ defined by $I_2(\varphi)$ is the graph of the embedding of the degree $d$ rational normal curve into projective space. Let $f_{i,j}$ denote the $2\times 2$ minor of $\varphi$ involving columns $i$ and $j$ (starting at 0) so that $I_2(\varphi) = \ideal{f_{i,j}\, |\, 0\leq i<j\leq d}$. Note that $\depth I_2(\varphi) \leq 2$ since $\ideal{x_0,x_1}$ is a codimension 2 associated prime of $I_2(\varphi)$ (so Theorem A2.10c from \cite{Eisenbud2004} does not apply). The Eagon-Northcott complex of $\varphi$ is given by
\begin{center} \tiny
\begin{tikzcd}
    \mc{C}^0: 0 \ar[r] &[-0.5em] 
    \ds\bigoplus_{j=1}^d S(-jd,-(d+1-j))^{\binom{d+1}{d+1}} \ar[r] &[-1em]
    % \bigoplus_{j=1}^{d-1} S(-jd,-(d-j))^{\binom{d+1}{d}} \ar[r] &
    \cdots \ar[r] &[-1em]
    \ds\bigoplus_{j=1}^i S(-jd,-(i+1-j))^{\binom{d+1}{i+1}} \ar[r] &[-1em]
    \cdots \ar[r] &[-1em]
    S(-d,-1)^{\binom{d+1}{2}} \ar[r,"I_2(\varphi)"] &[-0.5em] S
\end{tikzcd}
\end{center}
and we wish to show that this is a virtual resolution of $S/I_2(\varphi)$ which is not a free resolution.

Let $J \subseteq S$ be the ideal
\[ J = I_2 \left(\begin{bmatrix} x_0 & y_0 & \cdots & y_{d-1}\\ x_1 & y_1 & \cdots & y_d\end{bmatrix} \right)\]
which also defines the graph of the degree $d$ rational normal curve, and let $g_{i,j}$ for $0\leq i<j\leq d$ be the $2 \times 2$ minors defining $J$. The following relations show that $J \subseteq (I_2(\varphi):\ideal{x_0,x_1}^\infty) \subseteq (I_2(\varphi):B^\infty)$.
\begin{align*}
    x_0^{d+j-2}\cdot g_{0,j} &= x_0^{j-i}f_{0,j}-x_0^{j-2}x_1f_{0,j-1} & \text{for  } 1\leq j\leq d\\
    x_0^{d-i+2}x_1^{i-1}\cdot g_{i,j} &= x_0y_{i-1}f_{i-1,j}-x_1y_{i-1}f_{i-1,j-1}-x_0y_{j-1}f_{i-1,i} \qquad & \text{for  } 1\leq i<j\leq d
\end{align*}
Since $J$ is an ideal defining a rational normal scroll, we know that $S/J$ is Cohen-Macaulay, so $\depth J = d$. Therefore, $\depth (I_2(\varphi):B^\infty)\geq d$, and thus, Remark \ref{remark} implies that each of the generalized Eagon-Northcott complexes $\mc{C}^i$ for $-1\leq i\leq d$ is a virtual resolution. In particular, the Eagon-Northcott complex $\mc{C}^0$ is a virtual resolution of $S/I_2(\varphi)$ that is not exact since the relation $x_0^2f_{1,2}-x_0x_1f_{0,2}+x_1^2f_{0,1}=0$ gives a nonzero element of $H_1(\mc{C}^0)$.
\end{example}

\bibliography{bib}{}

\end{document}